\documentclass[12pt, a4paper, reqno]{amsart}
\usepackage[russian,english]{babel}
\usepackage{amsmath,amsfonts,amssymb,amsthm,graphics,epsfig}
\usepackage{hyperref}
\usepackage{graphicx,graphics}
\usepackage{epstopdf}
\usepackage{color}

\usepackage[margin=2cm]{geometry}

\newtheorem{theorem}{Theorem}[section]
\newtheorem{lemma}[theorem]{Lemma}

\newtheorem{corollary}[theorem]{Corollary}

\theoremstyle{definition}
\newtheorem{definition}[theorem]{Definition}

\theoremstyle{remark}

\newtheorem{example}[theorem]{Example}

\numberwithin{equation}{section}

\newcommand { \ib }[1] {\textit{\textbf{#1}}}

\begin{document}
\renewcommand{\ib}{\mathbf}
\renewcommand{\proofname}{Proof}
\renewcommand{\phi}{\varphi}
\newcommand{\conv}{\mathrm{conv}}
\makeatletter \headsep 10 mm \footskip 10 mm

\newcommand{\Zivaljevic}{\v{Z}ivaljevi\'{c}}

\title{Another ham sandwich in the plane}

\author{Alexey~Balitskiy}

\address{Dept. of Mathematics, Moscow Institute of Physics and Technology, Institutskiy per. 9, Dolgoprudny, Russia 141700}
\address{Institute for Information Transmission Problems RAS, Bolshoy Karetny per. 19, Moscow, Russia 127994}

\email{alexey\_m39@mail.ru}

\author{Alexey~Garber}

\address{Faculty of Mechanics and Mathematics, Moscow State University, Moscow, Leninskie gory, 1, Russia, 119991}
\address{B.N.~Delone International Laboratory ``Discrete and Computational Geometry'', Yaroslavl' State University, Sovetskaya st. 14, Yaroslavl', Russia 150000}

\email{alexeygarber@gmail.com}

\thanks{The work of A.~Garber is supported by the Russian Foundation of Basic Research grants 11-01-00633-a and 11-01-00735-a, and the Russian government project 11.G34.31.0053.}

\author{Roman~Karasev}

\address{Dept. of Mathematics, Moscow Institute of Physics and Technology, Institutskiy per. 9, Dolgoprudny, Russia 141700}
\address{Institute for Information Transmission Problems RAS, Bolshoy Karetny per. 19, Moscow, Russia 127994}
\address{B.N.~Delone International Laboratory ``Discrete and Computational Geometry'', Yaroslavl' State University, Sovetskaya st. 14, Yaroslavl', Russia 150000}

\email{r\_n\_karasev@mail.ru}
\urladdr{http://www.rkarasev.ru/en/}

\thanks{The work of R.~Karasev is supported by the Dynasty foundation, the President's of Russian Federation grant MD-352.2012.1, and the Russian government project 11.G34.31.0053.}


\begin{abstract}
We show that every two nice measures in the plane can be partitioned into equal halves by translation of an angle from arbitrary $k$-fan when $k$ is odd and in some cases when $k$ is even. We also give some counterexamples for certain fans and measures.

\end{abstract}

\maketitle

\section{Introduction}

In \cite[Theorem~1]{ziv2013} Rade~\Zivaljevic\ proved that it is possible to cut a half of every one of $d$ nice measures in $\mathbb{R}^d$ by a union of several cones of a simple fan translated by some vector. A simple fan is built of cones on facets of a simplex and a union of several cones is called a ``curtain'' in~\cite{ziv2013}. Throughout this paper we call a measure $\mu$ in $\mathbb R^d$ \emph{nice} if it is normalized with $\mu(\mathbb R^d) = 1$, has compact support, and every affine hyperplane $H$ has $\mu(H) = 0$.

In particular, in the two-dimensional case the curtain partition theorem can be reformulated as follows: For every $3$-fan $\mathcal{F}=\{F_1,F_2,F_3\}$ and for every two nice measures $\mu_1$ and $\mu_2$ in $\mathbb{R}^2$ there are index $i$ and translation $\mathbf{t}+F_i$ of an angle of $\mathcal{F}$ such that 
$\mu_1(\mathbf{t}+F_i)=\mu_2(\mathbf{t}+F_i)=\frac12.$ Here by an \emph{angle} we mean a set consisting of a vertex, two rays from this vertex, and everything between the rays. A particular case of an angle will be a halfplane.

In this note we generalize this theorem for the case of $k$-fans for odd $k$:

\begin{theorem}
\label{odd-fans}
Let $k\ge 3$ be odd. For every $k$-fan $\mathcal{F}=\{F_1,F_2,\ldots, F_k\}$ and for every two nice measures $\mu_1$ and $\mu_2$ in $\mathbb{R}^2$ there are index $i$ and a translation $\mathbf{t}+F_i$ of an angle of $\mathcal{F}$ such that 
$$
\mu_1(\mathbf{t}+F_i)=\mu_2(\mathbf{t}+F_i)=\frac12.
$$
\end{theorem}

And for some symmetric fans:

\begin{theorem}
\label{symm-fans}
Theorem~\ref{odd-fans} remains true for centrally symmetric fans of $2k$ angles assuming that $k$ is even.
\end{theorem}

The proofs of these planar theorems are rather elementary by a simple continuity argument, not requiring the Borsuk--Ulam type theorems like in~\cite{steinhaus1945,stone1942,ziv2013}. What is more interesting, it is possible to find counterexamples in the cases not covered by these theorems, see Section~\ref{section:examples}. In the last Section~\ref{section:noncompact} we show that it is impossible to cut half of two measures with non-compact support with translation of an angle in general case, but some generalization of our results in that case is still possible.


\subsection*{Acknowledgments.}
The authors thank Alfredo Hubard and the unknown referee for useful remarks and corrections.

\section{The proofs}

In what follows we make the standard assumptions: The measures are taken to have continuous densities and their supports are assumed to be compact and connected. The general case follows by a standard compactness argument, that we outline as follows. Let all supports of the measure be contained in a big disk $D$. Then for every measure $\mu_i$, by a convolution with a smooth kernel and a slight modification, we find a sequence of measures $\mu_i^{k}$, having continuous densities and connected supports in $D$, and converging to $\mu_i$. The convergence is in the sense that for any angle $A$ we have $\mu_i^{k}(A) \to \mu_i(A)$. Now for every $k$ we have an equipartitioning angle $A_k$, and we may assume that the intersections $A_k\cap D$ converge to an intersection $A\cap D$ in the Hausdorff metric, where $A$ is some angle. Assuming the contrary, without loss of generality  $\mu_i(A) < \frac12$, we see that for a small $\varepsilon$-neighborhood $A_\varepsilon$ of $A$ the inequality $\mu_i(A_\varepsilon) < \frac12$ retains. Here we use that $\mu_i$ of the boundary $\partial A$ and translates of this boundary slightly outside $A$ is zero. Now, for sufficiently large $k$, the convolution kernel's support may be assumed to lie in the $\varepsilon$-neighborhood of zero, and therefore we obtain $\mu_i^k(A_\varepsilon) < \frac12$. Also, for sufficiently large $k$, the angle $A_k$ is inside $A_\varepsilon$ and therefore $\mu_i^k(A_k) < \frac12$, which is a contradiction.

Under the above assumption of continuous density and connected support there is a unique halving line of given direction for any one of the measures. Let $\mathbf{e}_j, j=1,\ldots,k$ be vectors collinear to the rays of the fan $\mathcal{F}$ so that the angle $F_j$ lies between rays collinear to $\mathbf{e}_j$ and $\mathbf{e}_{j+1}$ (counter-clockwise direction and cyclic indices).

\begin{definition}
Denote $S^i_j$ the set of all points $M$ such that $\mu_i(\overrightarrow{OM}+F_j)=\frac12.$ We will call this set {\it the set of $j$-th color for the measure $\mu_i$}.
\end{definition}

\begin{lemma}
The set $S^i_j$ contains two infinite rays $r^i_{j,1}$ and $r^i_{j,2}$ with directions of vectors $-\ib e_j$ and $-\ib e_{j+1}$ that are connected by a continuous curve $\gamma^i_j$ in $S^i_j$. Moreover, we have $r^i_{j,2}=r^i_{j+1,1}$.
\end{lemma}

\begin{proof}
Let $l^i_j$ be the line of direction $\ib e_j$ that divides the measure $\mu_i$ into equal parts and $M$ an arbitrary point on this line. The function $$f(\alpha)=\mu_i(\overrightarrow{OM}- \alpha \ib e_j +F_j)$$ of positive real $\alpha$ is continuous and equal to $\frac12$ for sufficiently large $\alpha$ because the measure $\mu_i$ has compact support. This construction proves the existence of a ray $r^i_{j,1}.$ 

By the same reason we can find a suitable ray $r^i_{j-1,2}$ on the same line $l^i_j$ and take joint part of two rays as desired $r^i_{j,1}=r^i_{j-1,2}.$

And the last point of this lemma (existence of a curve) is also obvious since for every prescribed projection of $M$ onto the direction perpendicular to the bisector of $F_j$ there is a unique such $M$ in $S^i_j$ and this $M$ depends continuously on the projection.
\end{proof}

\begin{definition}
Let $\alpha^i_j$ be the angle between rays $r^i_{j,1}$ and $r^i_{j,2}$. These rays could have different starting points, but in that case we can take two lines containing them and choose as $\alpha^i_j$ the angle with sides that have unbounded intersections with rays $r^i_{j,1}$ and $r^i_{j,2}$.  This angle is obtained from $F_j$ by a central symmetry and a translation.

If rays $r^i_{j,1}$ and $r^i_{j,2}$ have opposite directions then construction of $\alpha^i_j$ is obvious and it is also obtained from $F_j$ by a central symmetry and translation.
\end{definition}

\begin{proof}[Proof of Theorem~\ref{odd-fans}]
Assume that for some $j$ the angle $\alpha^1_j$ is not a subset of the angle $\alpha^2_j$ and $\alpha^2_j$ is not a subset of $\alpha^1_j$. Since these angles are translates of each other, this assumption means that their boundaries intersect and either they have a whole ray of intersection, or they intersect transversally precisely once. Then two curves $r^1_{j,1}\cup r^1_{j,2} \cup \gamma^1_j$ and $r^2_{j,1}\cup r^2_{j,2} \cup \gamma^2_j$ have nonempty intersection too. Indeed, they differ from the boundaries $\partial \alpha^1_j$ and $\partial \alpha^2_j$ by modifications on compact parts. If the boundaries had an infinite ray of intersections then these curves will do the same. If the intersection was transversal and unique then these new curves will have an odd intersection index, and therefore have nonempty intersection.

So any point 
$$
X\in (r^1_{j,1}\cup r^1_{j,2} \cup \gamma^1_j)\cap (r^2_{j,1}\cup r^2_{j,2} \cup \gamma^2_j)
$$ 
can serve as the desired vertex of a translated angle $F_j$.

Now we may assume that for every $j$ one of $\alpha^1_j$ and $\alpha^2_j$ strictly contains the other. Without loss of generality, take some $j$ and assume that $\alpha^1_j\subset \alpha^2_j$. Then boundaries of these angles are not intersecting. Now we trace their ``left'' boundaries towards the infinity and note that $\partial\alpha^1_j$ is to the ``right'' of $\partial \alpha^2_j$. Here ``left'' and ``right'' are brief substitutes of ``anticlockwise'' and ``clockwise''. But the ``left'' parts of $\partial\alpha^1_j$ and $\partial \alpha^2_j$ are also the ``right'' parts of $\partial\alpha^1_{j+1}$ and $\partial \alpha^2_{j+1}$ respectively; hence $\alpha^2_{j+1}$ turns out to be a subset of $\alpha^1_{j+1}$. 

So we have shown that if $\alpha^1_j$ contains $\alpha^2_j$ then $\alpha^2_{j+1}$ contains
$\alpha^1_{j+1}$ and vice versa. But this relation is impossible to extend over the cyclic order since $k$ is odd.
\end{proof}

\begin{proof}[Proof of Theorem~\ref{symm-fans}]
The proof follows the same lines. We also conclude that the inclusions $\alpha^1_j\subset \alpha^2_j$ and $\alpha^2_j\subset \alpha^1_j$ alternate when $j$ goes around the circle. Since $k$ is even, we have, without loss of generality, 
\begin{equation}
\label{eq:opposite}
\alpha^1_1\subset\alpha^2_1\quad\text{and}\quad \alpha^1_{k+1}\subset \alpha^2_{k+1}.
\end{equation}
But the angles $\alpha^1_1$ and $\alpha^1_{k+1}$ share the common vertex and are opposite to each other, because they both are defined by the halving lines of $\mu_1$ in two directions. The same applies to $\alpha^2_1$ and $\alpha^2_{k+1}$ and we conclude that (\ref{eq:opposite}) is impossible.
\end{proof}

The above proof actually works in a slightly more general case:

\begin{theorem}\label{opp-rays}
The theorem $\ref{symm-fans}$ remains true for a possibly non-symmetric fan, that contains two opposite rays with even number of angles on one side of these rays.
\end{theorem}

\begin{proof}
The same argument as for the symmetric fans with $4k$ angles works for this case too. Assume that vectors $\ib e_i$ and $\ib e_{i+2m}$ has opposite directions. Without loss of generality we can assume that 
\begin{equation}
\label{eq:opposite2}
\alpha^1_i\subset\alpha^2_i\quad\text{and}\quad \alpha^1_{i+2m}\subset \alpha^2_{i+2m}.
\end{equation}
But the angles $\alpha_i^1$ and $\alpha_{i+2m}^1$ share the common line and lie on opposite sides of this line. The same applies to $\alpha_i^2$ and $\alpha_{i+2m}^2$, and the construction from \ref{eq:opposite2} is impossible.
\end{proof}

\section{Examples of measures and fans}
\label{section:examples}

\begin{example}\label{4fan} We start from an example of two measures and a 4-fan shown on the figure \ref{pict:4fan}. Each red point (small disc) represents one third of the measure $\mu_1$, they are located in the midpoints of the regular triangle represented by dashed lines. And each blue point represents one third of the measure $\mu_2$, here one point lies in the center of the red triangle. One can easily check that no angle from the fan on the right of Figure~\ref{pict:4fan} can divide each measure on two equal parts after a translation (all angles of this fan are parallel to sides of the dashed triangle on the left part).

Indeed, for the top angle we observe that once its copy contains a blue point then it already contains two red points in its interior. For the angle to the left (and the one to the right), we observe that once it touches at least two red points then it already contains in its interior some two of the blue points. For the angle at the bottom we observe that once it touches some two of the blue points then it already contains two red points in its interior.

\begin{center}
\begin{figure}[!ht]
\includegraphics[width=0.4\textwidth]{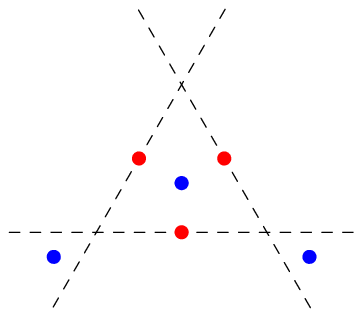}
\hskip 0.5cm
\includegraphics[width=0.4\textwidth]{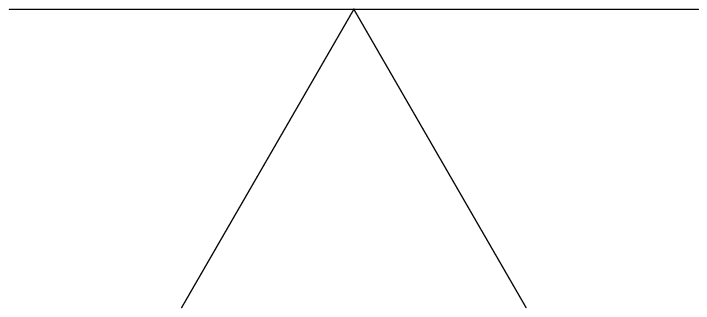}
\caption{Two measures and a $4$-fan without equipartitioning.}
\label{pict:4fan}
\end{figure}
\end{center}
\end{example}

Similarly we can construct example of $2k$-fan based on the regular $(2k-1)$-gon.

\begin{example}
The $2k$-fan consists of one angle equal to $\pi$ and $2k-1$ angles equal to $\frac{\pi}{2k-1}$. 
We start from a regular $(2k-1)$-gon and distribute the red measure $\mu_1$ in its midpoints of sides for even $k$ (see Figure~\ref{pict:4kfan}), or in the vertices for odd $k$ (see Figure~\ref{pict:4k2fan}). 

Let the blue measure $\mu_2$ have $k-1$ points in the ``inner'' region of the same $(2k-1)$-gon and additional $k$ points in the ``outer'' angles of the same $(2k-1)$-gon that lie below the ``bottom'' line. This is shown in Figures~\ref{pict:4kfan} and \ref{pict:4k2fan}. 

To explain this counterexample, we start with the top $\pi$ angle. Any of its translations touching more than a half of the blue measure must contain the ``bottom'' line of the $(2k-1)$-gon and therefore already contains at least $k$ red points in the interior. So this angle cannot make an equipartition.

Now consider the angle $F_1$ next to the top angle. If, after a translation, it touches more than a half of the red points then it must contain in the interior the $k-1$ ``inner'' blue points (it is seen in Figures~\ref{pict:4kfan} and \ref{pict:4k2fan} for both cases) and one more blue point to the left in the corresponding outer angle of the $(2k-1)$-gon. Because if an ``inner'' blue point is outside of it or on its boundary then the extension of one of its sides, passing almost through the center of the $(2k-1)$-gon, completely separates this angle from some $k$ of red points.

As for the next angle $F_2$, after a translation it must touch one of the ``inner'' blue points and at least one of the ``outer'' blue points (if it only touches ``outer'' blue points then it already contains everything). Then it is seen from the picture that this angle will also contain in its interior some $k$ red vertices of the $(2k-1)$-gon, because one of its sides is completely outside the red polygon, and the other cuts the bigger part of it. The argument for the remaining angles $F_j$ is the same depending on the parity of $j$.

\begin{center}
\begin{figure}[!ht]
\includegraphics[width=0.6\textwidth]{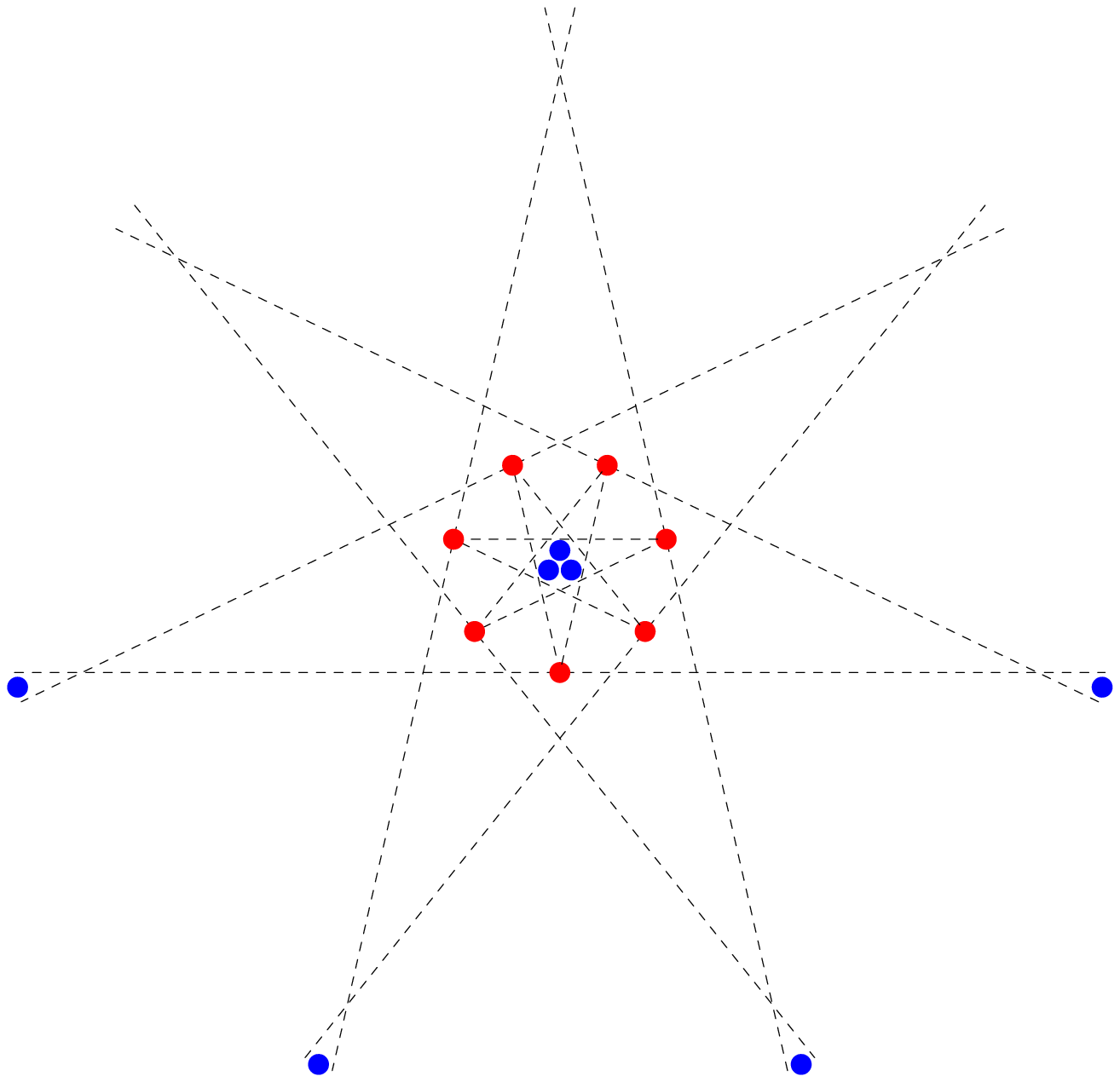}
\hskip 0.5cm
\includegraphics[width=0.3\textwidth]{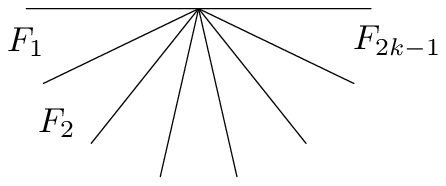}
\caption{Two measures and a $4m$-fan without equipartitioning.}
\label{pict:4kfan}
\end{figure}
\end{center}

\begin{center}
\begin{figure}[!ht]
\includegraphics[width=0.6\textwidth]{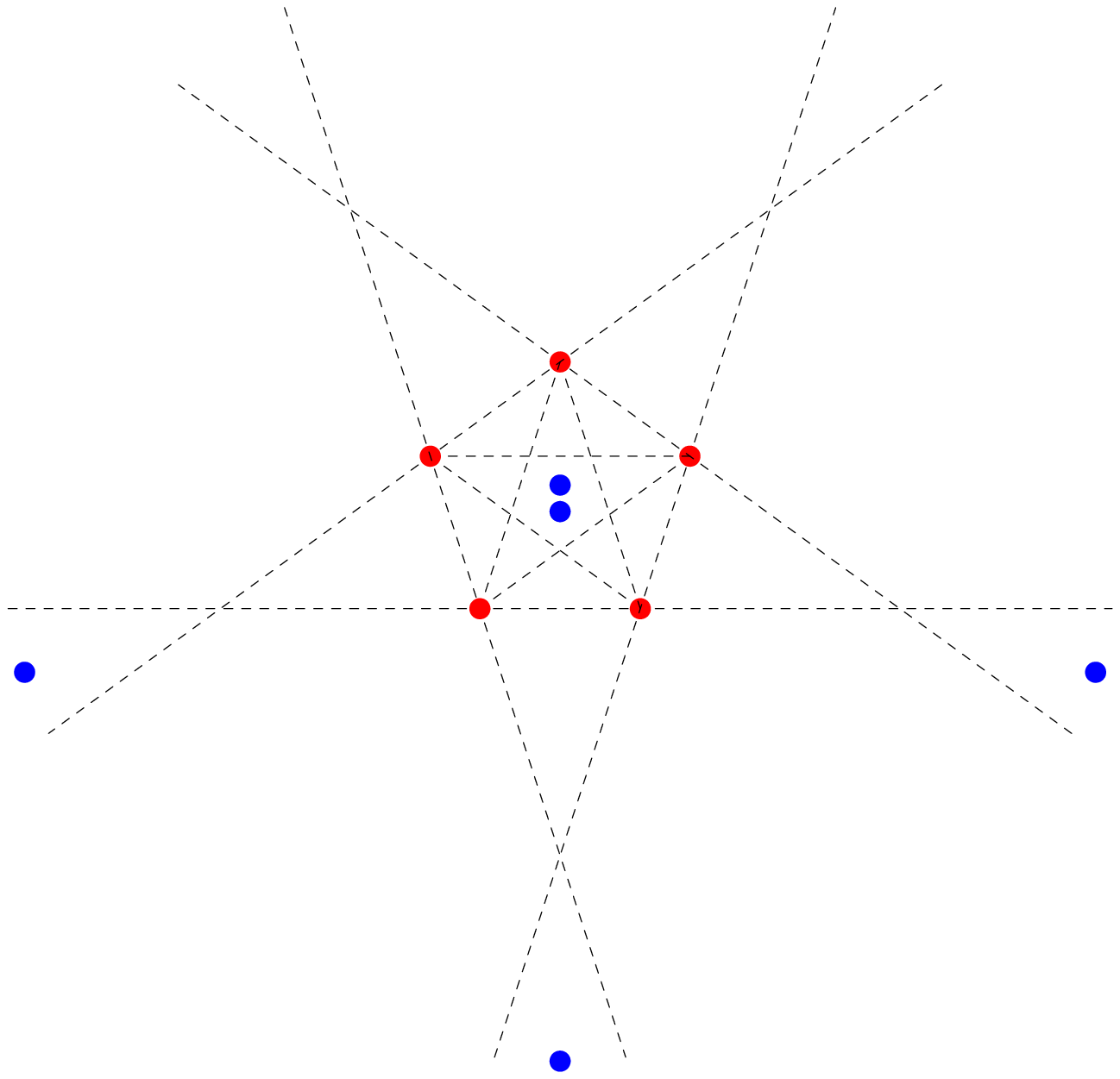}
\hskip 0.5cm
\includegraphics[width=0.3\textwidth]{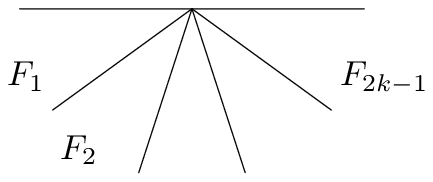}
\caption{Two measures and a $(4m+2)$-fan without equipartitioning.}
\label{pict:4k2fan}
\end{figure}
\end{center}
\end{example}

\begin{example}
Now we construct two measures for an arbitrary $4$-fan without two opposite rays that that will show that the conclusion of Theorem~\ref{opp-rays} fails for such fans (see Figure~\ref{pict:4fan-arb}). It is easy to establish by a small case analysis that, for any such fan,  it is possible to find angles $F_1$ and $F_3$ out of $F_1, F_2, F_3, F_4$ so that $-F_1$ is contained in $F_3$ and $F_3\setminus (-F_1)$ has two connected components. 

Now we split every measure into $3$ equal parts and do the following. One red point is put to the origin, two blue points are put in the respective components of $F_3\setminus(-F_1)$ near the origin, two red points are put in the respective components of $(-F_3)\setminus F_1$ much farther from the origin so that one of them is in $F_2$ and the other is in $F_4$, and we put the remaining blue point in $F_1$ very far from the origin.

Then we notice the following. If $F_3$, after a translation, touches two red points then it also contains the two blue points closer to the origin. If a translated $F_1$ touches two red points then it definitely contains in the interior one of the closer blue points and the farthest blue point. As for a translation of $F_2$ (and analogously $F_4$), if it touches some two of the blue points then it contains in the interior the red point at the origin and the other red point form the same side as $F_2$ as well.

\begin{center}
\begin{figure}[!ht]
\includegraphics[width=0.4\textwidth]{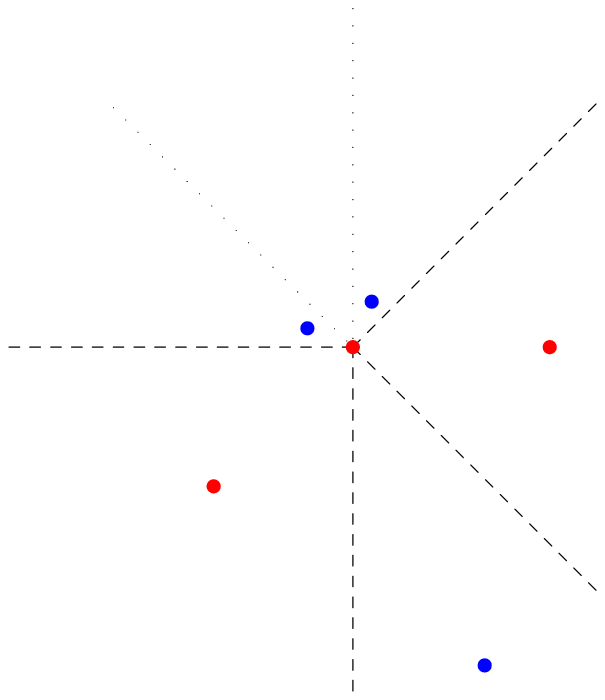}
\hskip 0.5cm
\includegraphics[width=0.4\textwidth]{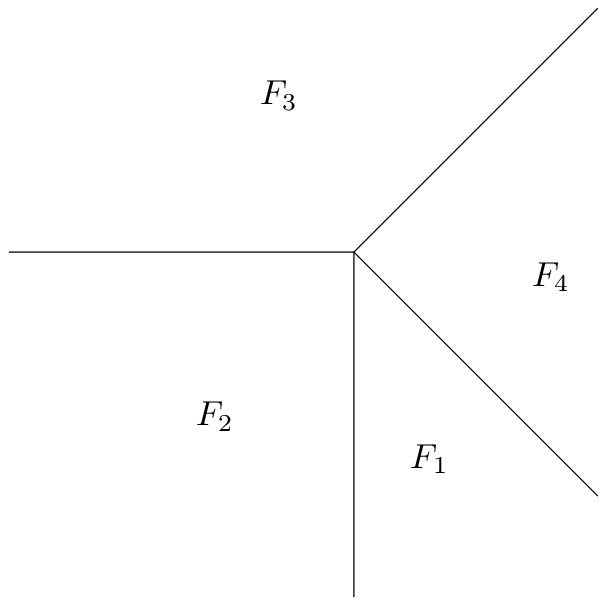}
\caption{Two measures and an arbitrary $4$-fan without equipartitioning.}
\label{pict:4fan-arb}
\end{figure}
\end{center}
\end{example}

\section{Measures with non-compact support}\label{section:noncompact}

Using the same approach as before we can extend Theorems \ref{odd-fans} and \ref{symm-fans} to the case of pairs of measures with non-compact support:

\begin{theorem}
Theorem \ref{odd-fans} holds for measures with non-compact support if no line parallel to any ray from the fan $\mathcal{F}$ divide both measures $\mu_1$ and $\mu_2$ into equal halves.
\end{theorem}

\begin{proof}
As in the proof of Theorem \ref{odd-fans}, construct the sets $S_j^i$ of all possible translations of the angle $F_j$ that divides $\mu_i$ into equal parts. If $\mu_i$ has non-compact support, this set does not necessarily contain two rays parallel to the sides of the angle $-F_j$, but it can be approximated by such rays when we go sufficiently far from the origin. Thus, as before we can construct similar angles $\alpha_j^i$ with sides on lines parallel to rays of the fan $\mathcal{F}$ that divides measures into equal parts.

Assume that for some $j$ the angle $\alpha^1_j$ is not a subset of the angle $\alpha^2_j$ and $\alpha^2_j$ is not a subset of $\alpha^1_j$. These angles does not have a common boundary ray, and we can find a point of intersection $S_j^1$ and $S_j^2$ using the same argument of odd index of intersection of a pair of curves. This completes the proof.

Another possible argument to prove this extension is to use going to the limit argument. The only case when we cannot choose a limit of a sequence of angles $A_k$ is when the intersection of $A_k$ with every disk $D_r$ tends to a half-plane. But in this case we would obtain a half-plane cutting precisely half of every measure, which is assumed to be impossible.
\end{proof}

Adding the same restriction we can modify the Theorem \ref{symm-fans} for measures with non-compact support. However, for arbitrary measures it is not always possible to find a translation of an angle that divides both measures into equal halves. We show this below for the case of two normally distributed measures.

\begin{lemma} No angle with angular measure less than $\pi$ can simultaneously cut a half of both measures $\mu_1$ and $\mu_2$ with densities $\rho_1(x,y)=\frac{1}{2\pi}\exp(-\frac{x^2+y^2}{2})$ and $\rho_2(x,y)=\frac{1}{8\pi}\exp(-\frac{x^2+y^2}{8})$. 
\end{lemma}

\begin{proof} Assume $F$ is a translation of the initial angle that has vertex in the origin $(0,0)$. Denote as $F_+(a,b)$ the translation of the angle $F$ by the vector $(a,b)$ and by $f_i(a,b)$ denote the $i$th measure of the angle $F_+(a,b)$, i.e. 
$$f_i(a,b):=\iint\limits_{F_+(a,b)}\rho_i(x,y)\;dxdy.$$

We will use the following properties of functions $f_1$ and $f_2$:
\begin{enumerate}
\item $f_1(a,b)=f_2(2a,2b)$.
\item If for some $i$ $f_i(a,b)=1/2$, then $(a,b)$ lies strictly inside the angle $-F$, centrally symmetric to $F$ with respect to the origin.
\end{enumerate}

The first property can be easily obtained by substituting $x=2x'$ and $y=2y'$ in the double integral formula for $f_2(2a,2b)$. Assume that the second property fails, then there is a half-plane $\pi$ containing the angle $F_+(a,b)$ and not containing the origin in its interior. So, $\mu_i(F_+(a,b))<\mu_i(\pi)\leq 1/2$, which is impossible.

Now, assume there is a vector $(a,b)$ such that $f_1(a,b)=f_2(a,b)=1/2$. Then $f_2(a,b)=f_2(2a,2b)=1/2$, but due to the second property, the angle $F_+(2a,2b)$ contains the angle $F_+(a,b)$ in its interior and equality $f_2(a,b)=f_2(2a,2b)$ is impossible.
\end{proof}

From the above proof it is clear that we could use two arbitrary normally distributed measures with common mean vector and proportional, but distinct, covariance matrices.

\begin{corollary}
If a fan $\mathcal{F}$ does not contain angle equal to $\pi$ then no translation of any angle from $\mathcal{F}$ can divide both measures $\mu_1$ and $\mu_2$ from the previous lemma into equal parts.
\end{corollary}

\end{document}